\documentclass[11pt,a4paper]{amsart}

\usepackage{amsthm}                   
\usepackage{amsmath}                  
\usepackage{amsfonts}                 
\usepackage{amssymb}                  
\usepackage{mathrsfs}                 
\usepackage{bbm}                      
\usepackage[margin=1in]{geometry}     
\usepackage{enumerate}								
\usepackage{graphicx}                 

\theoremstyle{definition}
\newtheorem{defi}{Definition}

\theoremstyle{plain}
\newtheorem{thm}[defi]{Theorem}

\newtheorem{cor}[defi]{Corollary}
\newtheorem{prop}[defi]{Proposition}

\numberwithin{equation}{section}

\newcommand{\R}{\ensuremath{\mathbbm{R}}}     
\newcommand{\C}{\ensuremath{\mathbbm{C}}}     
\newcommand{\Z}{\ensuremath{\mathbbm{Z}}}     

\def\lt{\left}
\def\rt{\right}

\newcommand{\Rd}{\ensuremath{\R^d}}                       
\renewcommand{\Re}[1]{\ensuremath{\mathrm{Re}\lt(#1\rt)}} 
\renewcommand{\Im}[1]{\ensuremath{\mathrm{Im}\lt(#1\rt)}} 
\newcommand{\Reh}[1]{\ensuremath{\mathcal{R}e(#1)}}       
\newcommand{\Imh}[1]{\ensuremath{\mathcal{I}m(#1)}}       
\newcommand{\M}{\ensuremath{\mathcal{M}}}                 
\newcommand{\set}[1]{\left\{#1\right\}}                   
\newcommand{\gen}[1]{\lt\langle #1\rt\rangle}             

\makeatletter
\def\imod#1{\allowbreak\mkern10mu({\operator@font mod}\,\,#1)} 
\def\@setcopyright{}                                           
\def\serieslogo@{}
\makeatother

\newcommand{\G}{\ensuremath{\Gamma}}        

\begin{document}
	\author[M.J.C.~Loquias]{Manuel Joseph C.~Loquias}
	\author[L.D.~Valdez]{Lilibeth D.~Valdez}
	\author[M.L.B.~Walo]{Ma.~Lailani B.~Walo}
	\address{Institute of Mathematics and Natural Sciences Research Institute, University of the Philippines Diliman, 1101 Quezon City, Philippines}
	\email{\{mjcloquias,ldicuangco,walo\}@math.upd.edu.ph}

	\title{Color groups of colorings of $N$-planar modules}

	\begin{abstract}
	 	A submodule of a $\Z$-module determines a coloring of the module where each coset of the submodule is associated to a unique color.  
		Given a submodule coloring of a $\Z$-module, the group formed by the symmetries of the module that induces a permutation of colors 
		is referred to as the color group of the coloring.  
		In this contribution, a method to solve for the color groups of colorings of $N$-planar modules where $N=4$ and $N=6$ are given.  
		Examples of colorings of rectangular lattices and of the vertices of the Ammann-Beenker tiling are given to exhibit how these methods may be extended to the general case.
	\end{abstract}



	\date{\today}

	\maketitle
	
	\section{Introduction}
		Colorings and color symmetries of crystals and quasicrystals can be considered a classical topic in crystallography (see \cite{S84,L97} and references therein).
		Most of the literature on colorings of lattices or $\Z$-modules put emphasis on Bravais colorings (\textit{cf.} \cite{B97b,BGS02,BG04,BDLPF13}).
		Nevertheless, studies have been made on general colorings of lattices or $\Z$-modules such as in \cite{H78,MP94,DLPF07}.
		In this contribution, we give some results on how to identify the color groups of sublattice colorings of lattices or $\Z$-modules,
		and give examples to illustrate them.  In particular, we give nonperfect non-Bravais colorings of the rectangular lattice and the vertices of the Ammann-Beenker tiling.
		
		By a \emph{$\Z$-module $\M$ of rank $r$ and dimension $d$}, we mean a subset of $\Rd$ that is spanned by $r$ vectors in $\Rd$ that 
		are linearly independent over $\Z$ but span $\Rd$.  
		As a group, $\M$ is isomorphic to 
		the free abelian group of rank $r$.  
		A $\Z$-module $\G$ of rank $r=d$ is called a \emph{lattice}. At the outset, we shall talk primarily of 
		lattices in dimension $d$, but results do hold for $\Z$-modules.  
		
		Let $\G_1$ be a lattice in $\Rd$.  A \emph{coloring} of $\G_1$ by $m$ colors is an onto mapping $c:\G_1\rightarrow C$, where $C$ is the 
		set of $m$ colors used in the coloring.  Of particular interest are colorings of $\G_1$ by $m$ colors wherein two points of $\G_1$ are assigned the 
		same color if and only if they belong to the same coset of a sublattice $\G_2$ of index $m$ in $\G_1$.  Such a coloring shall be referred to as a 
		\emph{coloring of $\G_1$ determined by its sublattice $\G_2$}. Here, the set of colors $C$ can be identified with the quotient group
		$\G_1/\G_2$ so that the color mapping $c$ is simply the canonical projection of $\G_1$ onto $\G_1/\G_2$ with kernel $\G_2$.  
		Hence, we may take $C=\{c_0=\mathbf{0},c_1,\ldots,c_{m-1}\}$ to be a complete	set of coset representatives of $\G_2$ in $\G_1$, and say that the coset 
		$c_j+\G_2$ has color $c_j$.
		
		Denote by $G$ the symmetry group of $\G_1$ and fix a sublattice coloring $c$ of $\G_1$. 
		The \emph{point group} of $G$, $P(G)$, consists of all symmetries in $G$ that fix the origin.		
		A symmetry in $G$ is called a \emph{color symmetry} of the
		coloring if all and only those points having the same color are mapped by the symmetry to a fixed color.  The set $H$ of all color symmetries of the coloring,
		that is, 
		$H=\{h\in G:\exists\,\sigma_h\in S_C\text{ such that }\forall\,\ell\in\G_1, c(h(\ell))=\sigma_h(c(\ell))\},$
		where $S_C$ is the group of permutations on the set of colors $C$, forms a group called the \emph{color group} or \emph{color symmetry group} of the coloring.
		The mapping $P:H\rightarrow S_C$ with $h\mapsto\sigma_h$ defines a group homomorphism, and thus the group $H$ acts on $C$.
		A coloring where $G=H$ is referred to as a \emph{perfect coloring}.
		
		The color groups of sublattice colorings of the square and hexagonal lattices were computed in~\cite{DLPF07}.  They were obtained by representing 
		sublattices of the square and hexagonal lattices as upper triangular matrices and expressing elements of the point groups of the lattices as matrices relative
		to the standard basis of~$\R^2$.		
		On the other hand, colorings of the square and hexagonal lattices determined by their square and hexagonal sublattices, respectively, (also referred to as 
		\emph{Bravais colorings}) were considered in~\cite{BDLPF13}.  In particular, the color groups were computed using number-theoretic properties of
		the ring of Gaussian and Eisenstein integers.
	
	\section{Some general results}
		Consider a coloring $c$ of the lattice $\G_1$ in $\Rd$ determined by a sublattice $\G_2$ of index $m$ in~$\G_1$.  The next theorem, which 
		belongs to the folklore and was proven for specific cases (\textit{cf.}~\cite[Theorems~1 and 2]{DLPF07} and~\cite[Theorem 2]{BDLPF13}), 
		gives us a method to compute for the color group $H$ of $c$.
		
		\begin{thm}\label{genresult}
			Let $G$ be the symmetry group of $\G_1$, where $P(G)$ and $T(G)$ denote the point group and translation subgroup of $G$, respectively.  
			Then $T(G)$ is a subgroup of the color group $H$.  
			In addition, an element $g\in P(G)$ is a color symmetry of $c$ if and only if $g$ leaves $\G_2$ invariant, that is, $g\G_2=\G_2$.			
		\end{thm}
		\begin{proof}
			Let $g\in T(G)$ be the translation by $\ell\in\G_1$. For any $c_i+\G_2\in \G_1/\G_2$, $g(c_i+\G_2)=(\ell+c_i)+\G_2=c_j+\G_2$ for some $0\leq j
			\leq m-1$.  Hence, $g$ sends color $c_i$ to color $c_j$, and $g\in H$.
				
			Suppose $g\in P(G)$  is a color symmetry of $c$.  Then $g$ sends color $c_0$ to some color $c_i$, $0\leq i\leq m-1$, that is,  
			$g\G_2=c_i+\G_2$.  Since $\mathbf{0}\in\G_2$ and $g(\mathbf{0})=\mathbf{0}$, then $g\G_2=\G_2$.
			In the other direction, suppose $g\G_2=\G_2$.  Then for every coset $c_i+\G_2$, $g(c_i+\G_2)=gc_i+g\G_2=c_j+\G_2$ for some $0\leq j\leq m-1$,
			since $g\in P(G)$.  Thus, $g\in H$.
		\end{proof}				
		
		It follows from above that every color symmetry of $c$ fixes color $c_0$. 
		Note however that Theorem~\ref{genresult} does not imply that the color group of a sublattice coloring and the symmetry group of the sublattice
		are the same.		
		
		The following corollary relates color groups of colorings determined by similar sublattices.
		\begin{cor}\label{primitive}
			Let $\G_3=\alpha R\G_2$ where $\alpha\in\R^+$ and $R\in O(d,\R)$.  
			Then $P(H')=R[P(H)]R^{-1}$ where $H'$ is the color group of the coloring of $\G_1$ determined by $\G_3$.		
		\end{cor}
		\begin{proof}
			Given $g\in P(G)$, $g\G_3=\G_3\Leftrightarrow gR\G_2=R\G_2\Leftrightarrow R^{-1}gR\G_2=\G_2$. 
			The claim now follows from Theorem~\ref{genresult}.
		\end{proof}
				
	\section{Color groups of sublattice colorings of planar lattices}
		Here, we view planar lattices as discrete subsets of $\C$.  
		In this setting, rotations about the origin by $\theta$ in the counterclockwise direction correspond to multiplication by the complex number $e^{i\theta}$.  
		We  denote such rotations by $R_{e^{i\theta}}$. 
		Also, the reflection $T_r$ along the real axis is associated to complex conjugation.  
	
		Suppose $\G_2$ is a sublattice of the planar lattice $\G_1$.  Then $R_{(-1)}\G_2=-\G_2=\G_2$.  
		By applying Theorem~\ref{genresult}, we obtain the following result.		
		\begin{prop}\label{180}
			The $180^{\circ}$-rotation about the origin is a color symmetry of a sublattice coloring of any planar lattice.
		\end{prop}
				
		\subsection{Sublattice Colorings of the Square and Hexagonal Lattices}
			We identify the square lattice $\G_1$ with the ring of Gaussian integers $\Z[i]$.  The symmetry group $G$ of $\G_1$ is of type $p4m$ (or $\ast442$ in 
			orbifold notation) and it is symmorphic, that is $G=P(G)\rtimes T(G)$ where $P(G)=\set{R_{\pm 1},R_{\pm i}} \rtimes \gen{T_r}$ and $T(G)$ is the 
			group of translations of $G$.  
			Given linearly independent vectors $u,v\in\G_1$, let $\G_2=\gen{u,v}_{\Z}$, that is, the sublattice of $\G_1$ generated by $u$ and $v$.   
			Applying Theorem~\ref{genresult} and Proposition~\ref{180} in this setting,
			it suffices to check whether $R_i$, $T_r$, and $R_iT_r$ fix the sublattice $\G_2$ in order to identify the color group of the coloring of $\Z[i]$ 
			determined by $\G_2$.  One obtains the following results.
			\begin{enumerate}[1.]
				\item $\G_2$ is invariant under $R_i$ if and only if $\Im{u\overline{v}}$ divides $N(u)$, $N(v)$, and $\Re{u\overline{v}}$. 
				
				\item $\G_2$ is invariant under $T_r$ if and only if $\Im{u\overline{v}}$ divides $\Im{u^2}$, $\Im{v^2}$, and $\Im{uv}$.
				
				\item $\G_2$ is invariant under $R_iT_r$ if and only if $\Im{u\overline{v}}$ divides $\Re{u^2}$, $\Re{v^2}$, and $\Re{uv}$.
			\end{enumerate}
			
			These observations yield the following theorem.
			
			\begin{thm}\label{sqr}
				Let $\G_1=\Z[i]$ and $\G_2=\gen{u,v}_{\Z}\subseteq\G_1$.  Then the color group of the coloring of $\G_1$ determined by $\G_2$ is of type
				\begin{enumerate}[\emph{(}a\emph{)}]
					\item $p4m$ if $\Im{u\overline{v}}$ divides $N(u)$, $N(v)$, $\Re{u\overline{v}}$, $u^2$, $v^2$, and $uv$.
					
					\item $p4$ if $\Im{u\overline{v}}$ divides $N(u)$, $N(v)$, and $\Re{u\overline{v}}$, and does not divide any of the following: $\Im{u^2}$, $\Im{v^2}$,
					and $\Im{uv}$.
					
					\item $pmm$ if $\Im{u\overline{v}}$ divides $\Im{u^2}$, $\Im{v^2}$, and $\Im{uv}$, and does not divide any of the following: $N(u)$, $N(v)$, and 
					$\Re{u\overline{v}}$.
					
					\item $cmm$ if $\Im{u\overline{v}}$ divides $\Re{u^2}$, $\Re{v^2}$, and $\Re{uv}$, and does not divide any of the following: $N(u)$, $N(v)$, and 
					$\Re{u\overline{v}}$.
					
					\item $p2$ if none of the above criteria is satisfied.
				\end{enumerate}
			\end{thm}
			
			If we write $\G_2=\gen{a,b+ci}_{\Z}$, where $a,b,c\in\Z$, then one readily obtains the results in~\cite{DLPF07} by applying Theorem~\ref{sqr}.  
			Because of Corollary~\ref{primitive}, we may even assume that $a$, $b$, and $c$ are coprime.  		
			
			Analogous results also hold for colorings of the hexagonal lattice.  
			We associate the hexagonal lattice to the ring of Eisenstein integers $\Z[\xi]$, where $\xi=e^{2\pi i/3}$.
			For the following result, given $u=a+b\xi\in\Z[\xi]$, we use the notation $\Reh{u}:=a$ and $\Imh{u}:=b$.
			\begin{thm}\label{hex}
				Let $\G_1=\Z[\xi]$ and $\G_2=\gen{u,v}_{\Z}\subseteq\G_1$.				
				Then the color group of the coloring of $\G_1$ determined by $\G_2$ is of type
				\begin{enumerate}[\emph{(}a\emph{)}]
					\item $p6m$ if $\Imh{u\overline{v}}$ divides $N(u)$, $N(v)$, $\Reh{u\overline{v}}$, $u^2$, $v^2$, and $uv$.
					
					\item $p6$ if $\Imh{u\overline{v}}$ divides $N(u)$, $N(v)$, and $\Reh{u\overline{v}}$, and does not divide any of the following: $\Imh{u^2}$, $\Imh{v^2}$,
					and $\Imh{uv}$.
					
					\item $cmm$ if $\Imh{u\overline{v}}$ divides $\Imh{u^2}$, $\Imh{v^2}$, and $\Imh{uv}$, and does not divide any of the following: $N(u)$, $N(v)$, and 
					$\Reh{u\overline{v}}$.
					
					\item $cmm$ if $\Imh{u\overline{v}}$ divides $\Reh{u^2}$, $\Reh{v^2}$, and $\Reh{uv}$, and does not divide any of the following: $N(u)$, $N(v)$, and 
					$\Reh{u\overline{v}}$.
					
					\item $cmm$ if $\Imh{u\overline{v}}$ divides $\Reh{u^2}-\Imh{u^2}$, $\Reh{v^2}-\Imh{v^2}$, and $\Reh{uv}-\Imh{uv}$, and does not divide any 
					of the following: $N(u)$, $N(v)$, and $\Reh{u\overline{v}}$.
					
					\item $p2$ if none of the above criteria is satisfied.
				\end{enumerate}
			\end{thm}
			
		\subsection{Example of colorings of a rectangular lattice}
			Let $\G_1$ be the rectangular lattice given by $\G_1=\langle 1, 2i \rangle_{\Z}$.  
			For Figure \ref{rect}(a), $\G_2=\langle 3, 2 i\rangle_{\Z}$.  
			By using the above results to check for the invariance of $\G_2$ under $T_r$, the color group of the coloring of $\G_1$ determined by $\G_2$ is of type $pmm$.  
			For Figure \ref{rect}(b),  $\G_3=\langle 4, 1+2i\rangle_{\Z}$.  
			Now, by applying the results above, the color group of the coloring of $\G_1$ determined by $\G_3$ is of type $p2$.
			
			\begin{figure}[ht]
				\centering
				\begin{minipage}{15pc}
					\centering
 					\includegraphics[width=7.85pc]{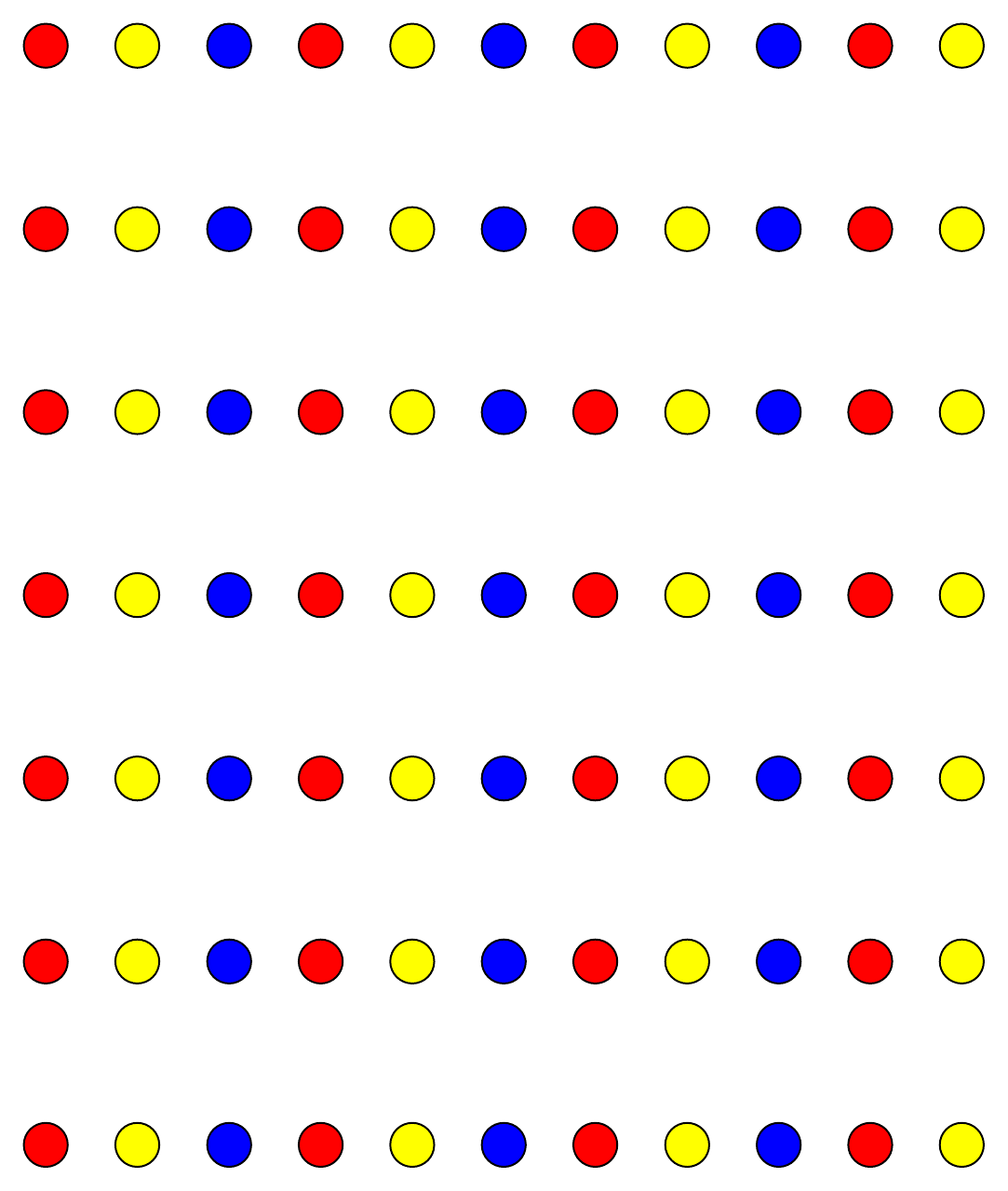}\\
					(a)					
				\end{minipage}\hspace{2pc}
				\begin{minipage}{15pc}
					\centering
					\includegraphics[width=7.85pc]{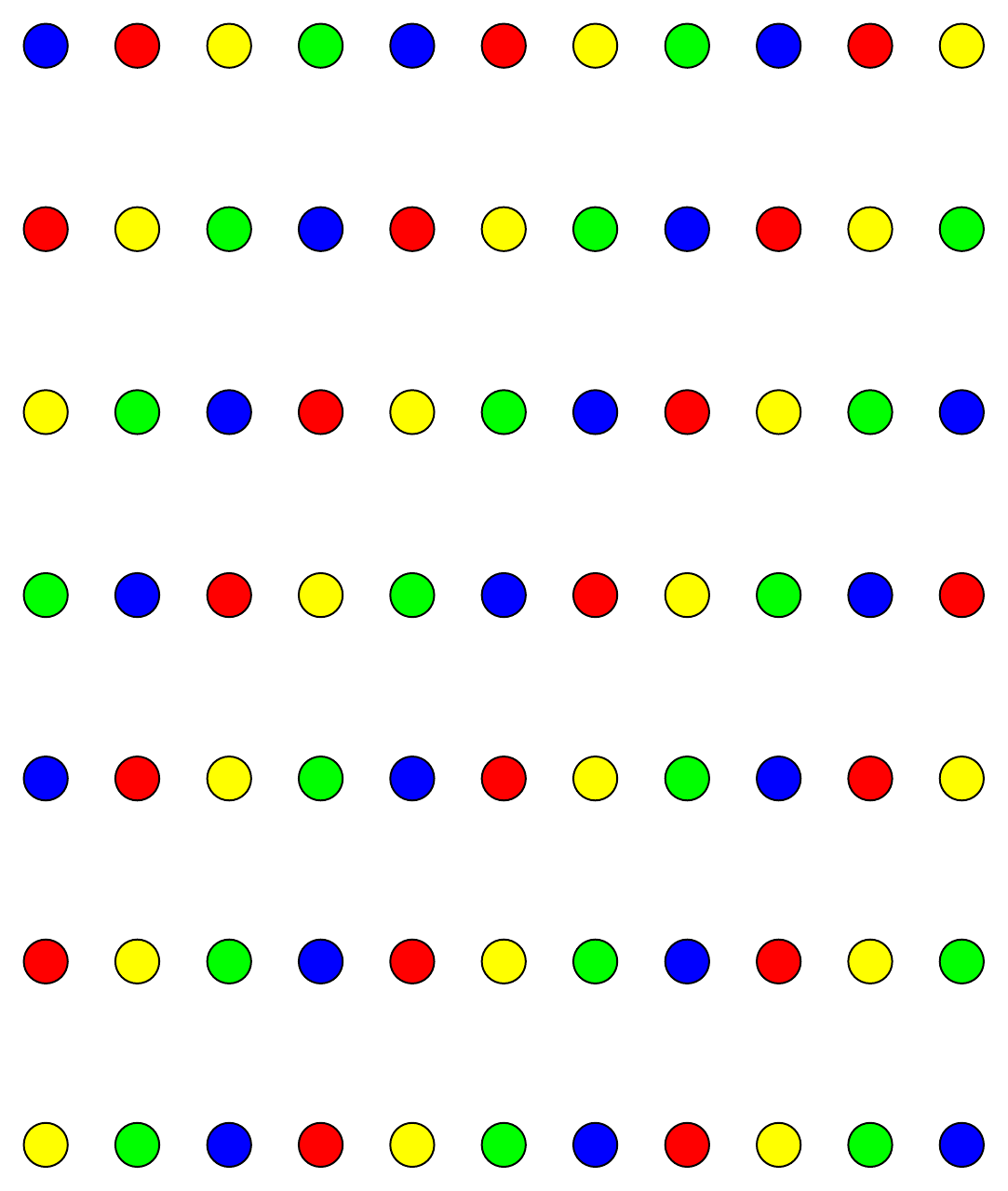}\\
					(b)					
				\end{minipage} 
				\caption{\label{rect}Colorings of $\G_1= \langle 1, 2i \rangle_{\Z}$ determined by (a) $\G_2=\langle 3,2 i \rangle_{\Z}$ and (b) $\G_3=\langle 4, 1+2i \rangle_{\Z}$.}
			\end{figure}
			
			\section{Nonperfect colorings of the Ammann-Beenker tiling}		
			The techniques that we have developed in the previous sections may also be used to obtain submodule colorings of $\Z$-modules.
			We illustrate this by considering the planar $\Z$-module $\M_1=\Z[\xi]$ of rank $4$ with basis $\set{1,\xi,\xi^2,\xi^3}$ 
			where $\xi=e^{\pi i/4}$.  
			The symmetry group of $\Z[\xi]$ is $G=\gen{a,b}\cong D_8$, where $a$ and $b$ correspond to the $45^\circ$-rotation in the counterclockwise direction about the origin and 
			reflection along the real axis, respectively.
			These submodule colorings induce colorings of the vertices of the eight-fold symmetric Ammann-Beenker tiling that are not necessarily perfect.
			
			Given a submodule $\M_2$ of $\M_1$, we identify the color group of  the coloring  determined by $\M_2$ by checking the invariance of $\M_2$ 
			under $R_{\xi}$,  $R_{\xi^2}$, $R_{\xi^3}$, and $T_r$. Note that $\M_2$ is always fixed by the $180^\circ$-rotation about the origin.
			
			As an illustration, to obtain a coloring of the Ammann-Beenker tiling with color group isomorphic to $D_4$, 
			we choose the basis elements of $\M_2$ such that $\M_2$ is invariant under $R_{\xi^2}$ and $T_r$ but not under $R_{\xi}$.  
			For example, if $\M_2 = \langle 1, 2\xi, \xi^2, 2\xi^3 \rangle$ then it can be verified that $\M_2$ satisfies the given conditions. 
			Thus, the color group of the coloring induced by $\M_2$ is $\gen{a^2,b}\cong D_4$ (see Figure \ref{AB}(a)).
			The submodule $\M_3 =\gen{1, 2\xi, 2\xi^2, \xi^3}$ is invariant only under the rotation by $180^\circ$. 
			The color group of the coloring induced by $\M_3$ is $\gen{a^4}\cong\Z_2$ (see Figure~\ref{AB}(b)).		
			\begin{figure}[ht]
				\centering
				\begin{minipage}{17pc}
					\centering
					\includegraphics[width=14pc]{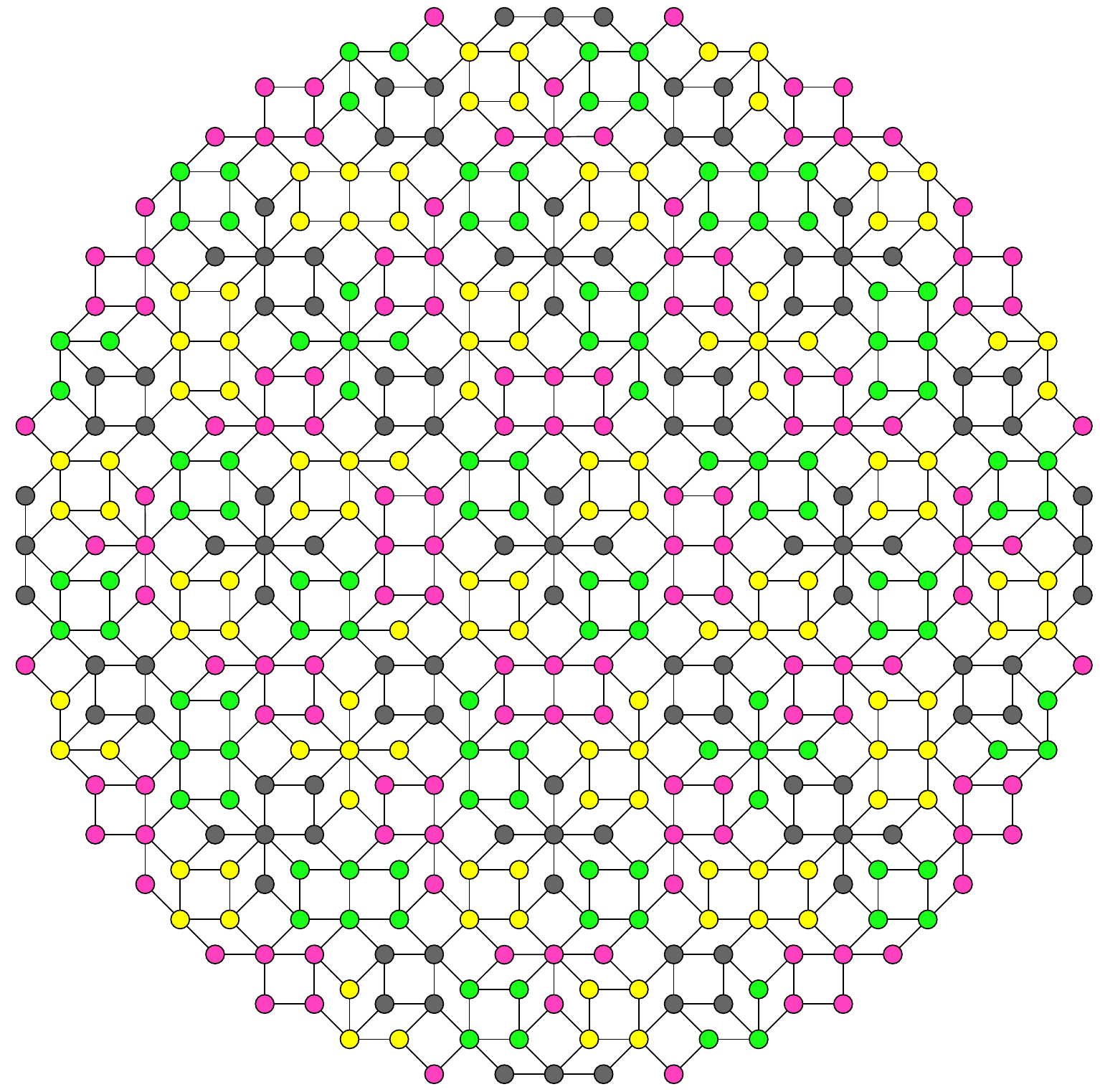}\\
					(a)
				\end{minipage}\hspace{2pc}
				\begin{minipage}{17pc}
					\centering
					\includegraphics[width=14pc]{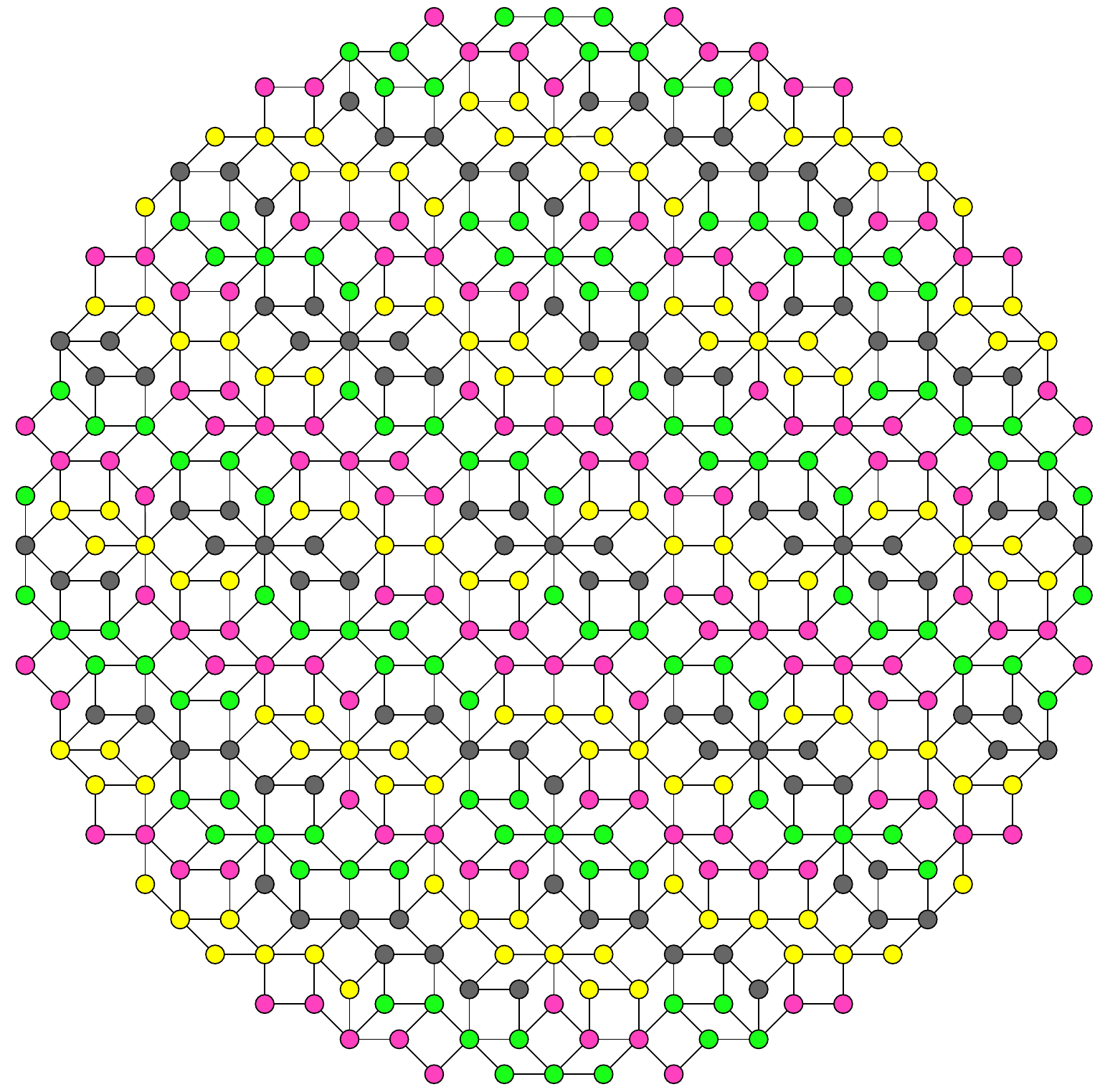}\\
					(b)
				\end{minipage}
				\caption{Colorings of the Ammann-Beenker tiling induced by (a) $\M_2 = \gen{1, 2\xi, \xi^2, 2\xi^3}$ and (b) $\M_3 = \gen{1, 2\xi, 2\xi^2, \xi^3}$}
				\label{AB}
			\end{figure}
							
	\subsection*{Acknowledgement} 
		The authors are grateful for the financial support by the Natural Sciences Research Institute (NSRI) of the University of the Philippines Diliman under
		Project Code: MAT-12-1-01.


\begin{thebibliography}{9}
		\bibitem{S84} Schwarzenberger R L E 1984 Colour symmetry {\it Bull. Lond. Math. Soc.} {\bf 16} 209-40	

		\bibitem{L97} Lifshitz R 1997 Theory of color symmetry for periodic and quasiperiodic crystals {\it Rev. Mod. Phys.} {\bf 69} 1181-1218
		
		\bibitem{B97b} Baake M 1997 Combinatorial aspects of colour symmetries {\it J. Phys.} A {\bf 30} 2687-98
		
		\bibitem{BGS02} Baake M, Grimm U and Scheffer M 2002 Colourings of planar quasicrystals {\it J. Alloys Compounds} {\bf 342} 195-97
		
		\bibitem{BG04} Baake M and Grimm U 2004 Bravais colourings of planar modules with $N$-fold symmetry {\it Z. Krist.} {\bf 219} 72-80
		
		\bibitem{BDLPF13} Bugarin E P, De Las Pe\~nas M L A  and Frettl\"oh D 2013 Perfect colourings of cyclotomic integers {\it Geom. Dedicata} {\bf 162} 271-82
		
		\bibitem{H78} Harker D 1978 Colored lattices {\it Proc. Natl. Acad. Sci. U.S.A.} {\bf 75} 5264-67	
		
		\bibitem{MP94} Moody R V and Patera J 1994 Colourings of quasicrystals {\it Can. J. Phys.} {\bf 72} 442-52
		
		\bibitem{DLPF07} De Las Pe\~nas M L A and Felix R 2007 Color groups associated with square and hexagonal lattices {\it Z. Krist.} {\bf 222} 505-12	
	\end{thebibliography}
\end{document}